\documentclass[12pt]{amsart}
\usepackage{amsmath}
\usepackage{amsthm,amssymb}
\usepackage{amsfonts}
\usepackage{graphicx}
\usepackage[all]{xy}
\usepackage{hyperref}

\usepackage{tikz}
\usetikzlibrary{decorations.markings}

\newcommand{\R}{\mathbb{R}}
\newcommand{\Z}{\mathbb{Z}}

\newcommand{\F}{\mathbb{F}}
\newcommand{\CP}{\mathbb{CP}}

\newcommand{\gen}[1]{\langle #1\rangle}

\DeclareMathOperator{\id}{I}
\DeclareMathOperator{\Cl}{C\ell}

\DeclareMathOperator{\G}{G}
\DeclareMathOperator{\abs}{abs}
\DeclareMathOperator{\cok}{cok}

\DeclareMathOperator{\GR}{\mathcal{GR}}

\newtheorem{theorem}{Theorem}[section]
\newtheorem{lemma}[theorem]{Lemma}
\newtheorem{proposition}[theorem]{Proposition}
\newtheorem{corollary}[theorem]{Corollary}
\theoremstyle{definition}
\newtheorem{definition}[subsection]{Definition}
\theoremstyle{remark}
\newtheorem{remark}[subsection]{Remark}
\newtheorem{example}[subsection]{Example}
\begin{document}

\title{The Signed Monodromy Group of an Adinkra}
\author[E. Goins]{Edray Goins}
\address{Pomona College,
	Department of Mathematics,
	610 North College Avenue,
	Claremont, CA 91711}
\email{edray.goins@pomona.edu}
\author[K. Iga]{Kevin Iga}
\address{
	Brown University,
	Brown Theoretical Physics Center,
	340 Brook St.,
	Providence, RI  02912
	and
	Pepperdine University,
	Natural Science Division,
	24255 Pacific Coast Hwy.,
  	Malibu CA 90263-4321
}
\email{kiga@pepperdine.edu}
\author[J. Kostiuk]{Jordan Kostiuk}
\address{Brown University,
	Box 1917,
	151 Thayer St.,
	Providence RI 02912}
\email{jordan\_kostiuk@brown.edu}
\author[K. Stiffler]{Kory Stiffler}
\address{
	Brown University,
	Brown Theoretical Physics Center,
	340 Brook St.,
	Providence, RI  02912}
\email{kory\_stiffler@brown.edu}
\date{September 3, 2019}
\maketitle
\begin{abstract}
An ordering of colours in an Adinkra leads to an embedding of this Adinkra into a Riemann surface $X$, and a branched covering map $\beta_X:X\to\CP^1$.   This paper shows how the dashing of edges in an Adinkra determines a signed permutation version of the monodromy group, and shows that it is isomorphic to a Salingaros Vee group.
\end{abstract}


\section{Introduction}
An Adinkra is a bipartite directed graph, together with various markings (each edge is coloured from among $N$ colours, and is either drawn with a solid or dashed line), subject to a certain list of conditions \cite{zhang}.  Adinkras arise from representations of the supersymmetry algebra from physics \cite{rA}.  In Ref.~\cite{geom1}, it was shown that given an Adinkra, and a cyclic ordering of the colours, there is an embedding of the Adinkra into a Riemann surface $X$, and a branched covering map $\beta_X:X\to \CP^1$, branched over the $N$th roots of $-1$.

One important approach in studying branched covers is the monodromy group.   This is the set of permutations of the points $\beta_X{}^{-1}(\infty)=\{p_1,\ldots,p_d\}$, resulting from loops in $\CP^1$ based at $\infty$ that avoid the branch points.  The monodromy group of this branched covering map turns out to be isomorphic to a group of the form $\F_2^m$, where $2^m=d$ \cite{geom1}.

The edges of an Adinkra are either solid or dashed.  This can be used to turn these permutations into signed permutations, meaning that we formally invent objects $\{-p_1,\ldots,-p_d\}$, and instead of $p_i$ going to $p_j$, it might go to $-p_j$, if there is an odd number of dashed edges that are involved in the corresponding path.  In this way we end up with a signed permutation group.

In this paper we calculate this signed permutation group for any connected Adinkra, and prove them to be the {\em Vee groups} $G_n$ due to Nikos Salingaros in Ref.~\cite{salingaros1}.

This will give a more general context for the appearance of the quaternionic group $Q_8$ (which is isomorphic to $\G_2$) in the case $N=4$ with code $\gen{1111}$, as described in \cite{Korrals}.

This distinction between solid or dashed edges corresponds naturally to the meaning that this distinction has in the representation of the $1$-dimensional supersymmetric Poincar\'e algebra, where Adinkras were first discussed.  The signed monodromy group also is the natural place for the {\em holoraumy tensors} $V_{IJ}$ and $\tilde{V}_{IJ}$ as defined in Ref.~\cite{HoLoRmY1,HoLoRmY2}.

We begin in Section~\ref{sec:review} by reviewing Adinkras, the Riemann surface $X$, the Bely\u{\i} map $\beta_X:X\to\CP^1$, and the monodromy group.  Section~\ref{sec:signed_perms} introduces the mathematics of signed permutations, and Section~\ref{sec:SMG} defines the signed monodromy group for an Adinkra.  

The signed monodromy group will be defined in Section~\ref{sec:SMG}, the Salingaros Vee group $\G_n$ will be defined in Section~\ref{sec:vee}, and the main theorem, which computes the signed monodromy group, will be stated in Section~\ref{sec:main}.  The proof of this will be in Sections~\ref{sec:proof1} through \ref{sec:proof2}.  There is an application to relations in the $\GR(d,N)$ algebra in Section~\ref{sec:relations}.

\section{Background}
\label{sec:review}
We begin by summarizing several concepts about Adinkras and the corresponding Riemann surfaces and Bely\u{\i} maps.  A more thorough introduction can be found in Refs.~\cite{rA,r6-1,at,zhang}.

\subsection{Adinkras}
To aid in studying $\GR(d,N)$ algebras, M. Faux and S. J. Gates introduced diagrams called {\em Adinkras} \cite{rA}.  
A mathematically rigorous definition of Adinkras is built in steps, as described in \cite{zhang}:
\begin{definition}
	\begin{itemize}
	\item An $N$-dimensional \emph{Adinkra topology} is a bipartite $N$-regular graph; we call the two sets in the bipartition \emph{bosons}  and \emph{fermions}, and colour them white and black respectively.
	\item An Adinkra \emph{chromotopology} is an Adinkra topology for which the set of edges are $N$-coloured, with each vertex incident with one edge of each colour, and the subset of edges consisting of two distinct colours forms a disjoint union of $4$-cycles---these special cycles are known as \emph{$2$-coloured $4$-cycles}.
	\item An \emph{Adinkra} is an Adinkra chromotopology equipped with two additional structures: an \emph{odd-dashing}---a dashing of the edges for which there is an odd number of dashed edges in each $2$-coloured $4$-cycle---and a \emph{height assignment}, which is a ranked poset structure on the vertices described by a $\mathbb{Z}$-valued function on the vertex set subject to certain constraints.
	\end{itemize}
\end{definition}

One example of an Adinkra is the $N$-dimensional Hamming cube, with vertex set $\{0,1\}^N=\F_2^N$.  
If $v$ and $w$ are vertices that are identical except in a single coordinate (say the $i$th coordinate), then there is an edge of colour $i$ connecting $v$ and $w$.  
Given a vertex $(x_1,\ldots,x_N)$, it is a boson if $\sum x_i$ is even, and it is a fermion if the sum is odd.  
The integer labeling of that vertex is given by $\sum x_i$.  
An edge of colour $i$ connecting $(x_1,\ldots,x_N)$ with $(y_1,\ldots,y_N)$ is solid if $\sum_{j=1}^{i-1} x_i$ is even and is dashed if the sum is odd (note that by definition, $x_j=y_j$ for all $j\not=i$).  See Figure~\ref{fig:adinkra40}.

Other interesting graphs can be obtained by quotienting the Hamming cube by some linear subspace of $\F_2^N$, also called a \emph{linear code}.

\begin{figure}
	\center 
	\includegraphics[width=\textwidth]{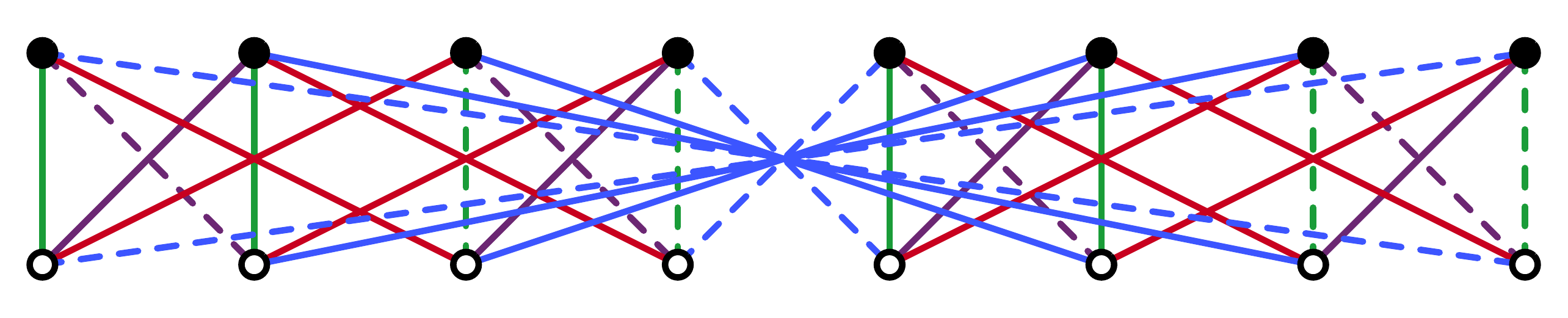}
\caption{An $N=4$ Adinkra with the topology of a Hamming $4$-cube.  Bosons are the white nodes, and fermions are black.  Note the colouring and dashing of the edges.  The height assignment here, shown literally by height on the page, puts all bosons at height $0$ and all fermions at height $1$.}
\label{fig:adinkra40}
\end{figure}

\begin{figure}
	\center 
\includegraphics[width=0.46\textwidth]{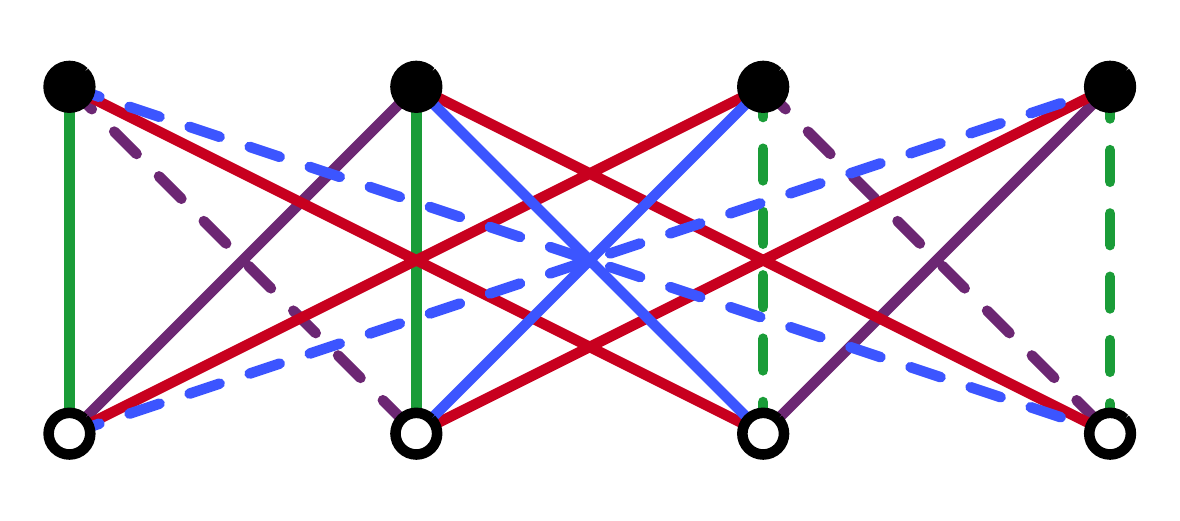}
\caption{The quotient of the $4$-cube Adinkra by the doubly even code generated by $\gen{1111}$.  This can be obtained from the Adinkra in Figure~\ref{fig:adinkra40} by identifying each of the four leftmost bosons (and the four leftmost fermions) with the boson (resp. fermion) four nodes to its right.}
\label{fig:adinkra41}
\end{figure}

A code is a vector space over $\F_2$, and so has a dimension, which we will call $k$.  The number of elements of the code will be $2^k$, and the number of vertices in the quotient is $2^{N-k}$.

This process preserves the bipartition of the vertices if the code is even, (meaning that the number of $1$s in every element is even) and can be made compatible with the dashing condition conditions above if the code is doubly even (meaning that the number of $1$s in every element is a multiple of four) \cite{at}.
That is, we can construct an Adinkra out a doubly even linear code. 
Conversely, every connected Adinkra arises in this fashion \cite{at}.
Since every Adinkra is a disjoint union of such connected components, the topology of Adinkras reduces to knowledge of the doubly even codes.  One example is the bottom Adinkra in Figure~\ref{fig:adinkra41}, which has $N=4$ and uses the code $\{0000,1111\}$.  This has $N=4$ and $k=1$, and is obtained by identifying each boson, and each fermion, with a corresponding boson (resp. fermion) that is diametrically opposite it on the $4$-cube.  In the Figure, this is obtained from the top Adinkra by identifying each of the four leftmost bosons (and the four leftmost fermions) with the boson (resp. fermion) four nodes to its right.

\subsection{The Riemann surface for an Adinkra}
\label{sec:belyi}
In this section we review the material in Ref.~\cite{geom1,geom2}.

Consider an Adinkra.  A {\em rainbow} is a cyclic ordering of the $N$ colours.  An Adinkra and a rainbow give rise to a Riemann surface $X$.  This is done by attaching a square to every cycle of four edges that alternate between two colours, both of which are adjacent in the rainbow.  The Riemann surface $X$ is connected if and only if the Adinkra is.  If $X$ is connected, its genus is
\[g=1+(N-4)\frac{d}{4},\]
where $d$ be the number of bosons (which is equal to the number of fermions).  For a connected Adinkra, which is the quotient of $\F_2^N$ by a code of dimension $k$,
\[d=2^{n-k-1}.\]

There is also an order $d$  branched covering map $\beta_X:X\to \CP^1$, branched over the $N$th roots of $-1$, $\xi_j=e^{\frac{(2j-1)\pi}{N}}$ for $j=1,\ldots, N$, and the covering map is ramified to order 2 around each of these $\xi_j$.  In $\CP^1$, we draw a single white vertex at $0$, a single black vertex at $\infty$, and $N$ edges from $0$ to $\infty$, one for each colour joining the two vertices.  These edges are parameterized by 
\[z_j(t)=te^{\frac{2(j-1)\pi i}{N}},\ 0\leq t\leq \infty,\]
and form the rays joining $0$ to $\infty$ making an angle $\frac{2(j-1)\pi}{N}$ with the real axis.  This $\CP^1$, together with these markings, we call a {\em beachball} (see Figure~\ref{fig:beachballn5}).

\begin{figure}
\begin{center}
	\includegraphics[width=3in]{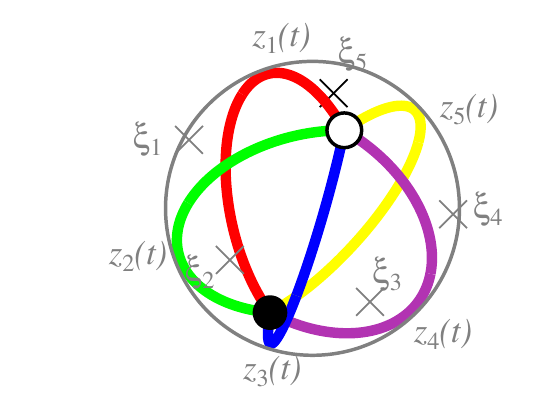}
\end{center}
\caption{The beachball for $N=5$.  The white node is the boson at $0$, the black node is the fermion at $\infty$, the coloured edges are the $z_j(t)$, and the $\times$ are at the $\xi_j$, where the covering map $\beta_X$ is ramified.}
\label{fig:beachballn5}
\end{figure}

\begin{figure}
\begin{center}
\includegraphics[width=2in]{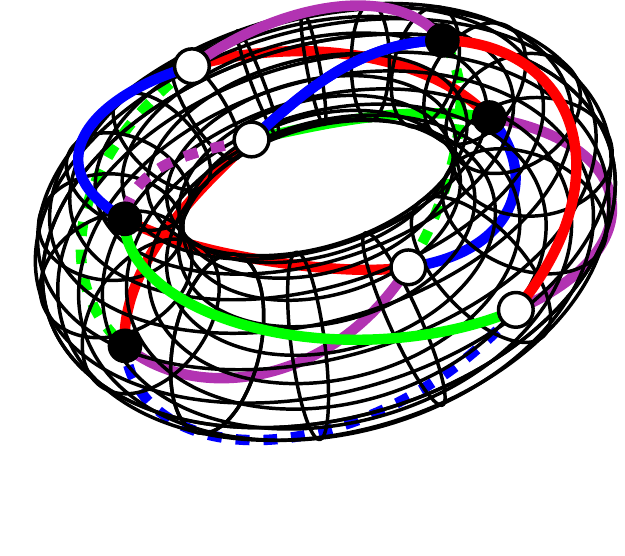}
\includegraphics[width=2in]{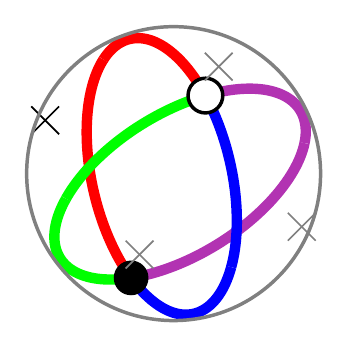}
\end{center}
\caption{On the left, the Riemann surface $X$ for the $N=4$, $k=1$ Adinkra from Figure~\ref{fig:adinkra41}.  There is an order $4$ branched covering map $\beta_X$ that sends it to the corresponding $N=4$ beachball on the right.}
\label{fig:torusbeachballn4}
\end{figure}

Then the preimage of the white vertex at $0$ by $\beta_X$ is the set of $d$ bosons, and likewise the preimage of the black vertex at $\infty$ by $\beta_X$ is the set of $d$ fermions.  The preimage by $\beta_X$ of the edge coloured $j$ is the set of edges of colour $j$.  The wedge between colour $j$ and colour $j+1$ in $\CP^1$ contains one branch point $\xi_j$ ramified to order $2$, and its preimage under $\beta_X$ is the set of squares between colour $j$ and colour $j+1$ edges.  Note that the map $\beta_X$ is not ramified on the Adinkra.

As an aside, we might mention that in Reference \cite{geom1}, this covering map $\beta_X$ was composed with the map $\beta_N:\CP^1\to\CP^1$ given by
\[\beta_N(z)=\frac{z^N}{z^N+1}.\]
The result, $\beta_N\circ\beta_X:X\to\CP^1$, is a cover of order $Nd$, branched over $\{0,1,\infty\}$, and is hence a {\em Bely\u{\i} map}, and connects the subject of Adinkras with the {\em dessins d'enfants} of Grothendieck.  But this composed map will not play a r\^ole in this paper, and we will focus exclusively on $\beta_X$.

\subsection{The Monodromy group}
Let $Y=\CP^1-\{\xi_1,\ldots,\xi_N\}$, the set of regular points of $\CP^1$ with respect to the branched cover $\beta_N$.  Pick the basepoint $y_0=\infty$ in $Y$ (the basepoint $y_0=0$ leads to a similar story).  Above $y_0=\infty$ lie all $d$ of the fermions in $X$, $F=\beta_X^{-1}(\{\infty\})$.

The fundamental group $\pi_1(Y,\infty)$ acts on the set $X_0$ as follows: for each loop $c\in\pi_1(Y,\infty)$
and each element $x_0\in F$, there is a unique lift of $c$ to a path in $X$ starting at $x_0$ and ending at some other fermion $\sigma_c(x_0)\in F$.
The point $\sigma_c(x_0)$ does not depend on the homotopy class of the loop $c$ and, in this way, we obtain a map $\sigma\colon\pi_1(Y,y_0)\to S_d$, where $S_d$ means the symmetric group on $d$ elements.

\begin{definition}
	Given a connected Adinkra $A$, the \emph{monodromy group of $A$} is the image of $\pi_1(Y,\infty)$ in $S_d$.
\end{definition}

\begin{figure}
\begin{center}
\includegraphics[width=2.8in]{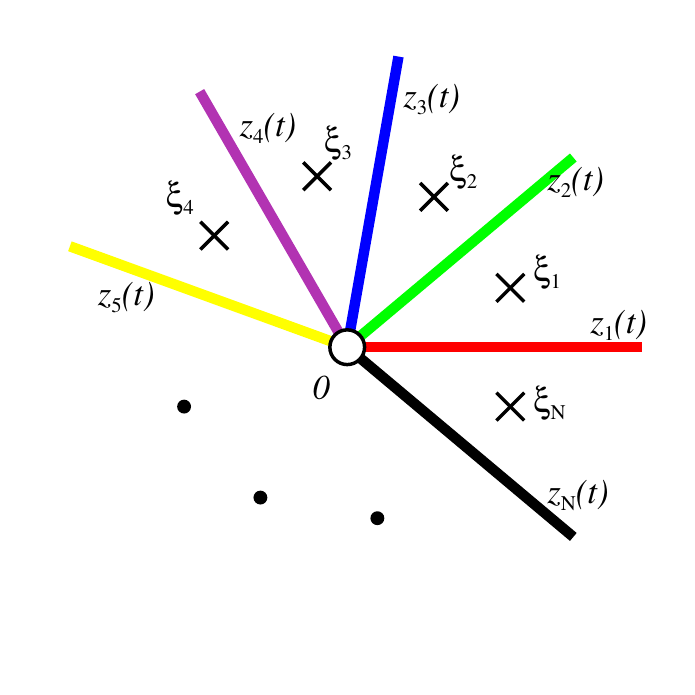}
\includegraphics[width=2.8in]{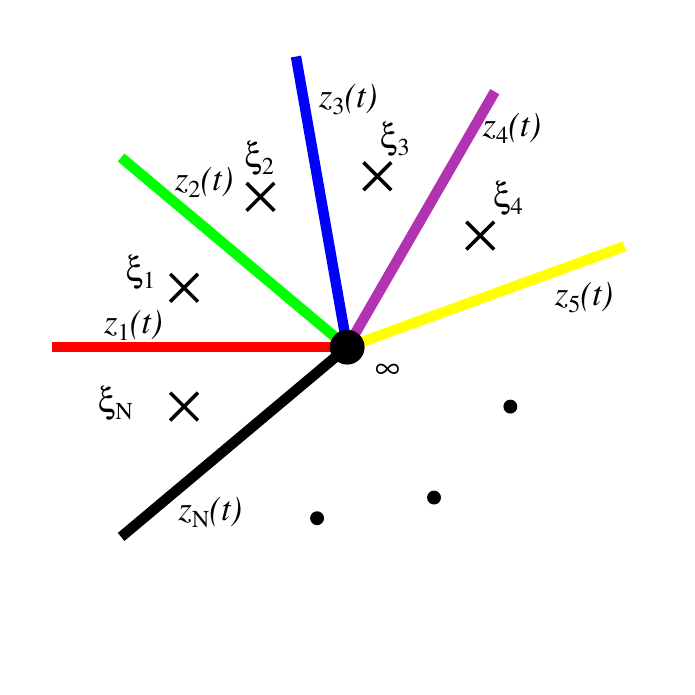}
\end{center}
\caption{Views of the beachball $\CP^1$ near $0$ and $\infty$.  The path $z_1(t)$ goes along the positive real axis.  Note that the ordering of colours goes counterclockwise around $0$ and clockwise around $\infty$.}
\label{fig:spokes-zero-infty}
\end{figure}

The fundamental group $\pi_1(Y,y_0)$ is generated by loops $c_j$ that start at $\infty$ and wrap around $\xi_j$ positively; the only relation between these loops is that\footnote{We will write path composition from right to left.}
\[c_1\cdots c_N=1,\]
so it suffices to use $c_1,\ldots,c_{N-1}$ as generators.  See Figure~\ref{fig:loops-c}.

\begin{figure}
\begin{center}
\includegraphics[width=2.5in]{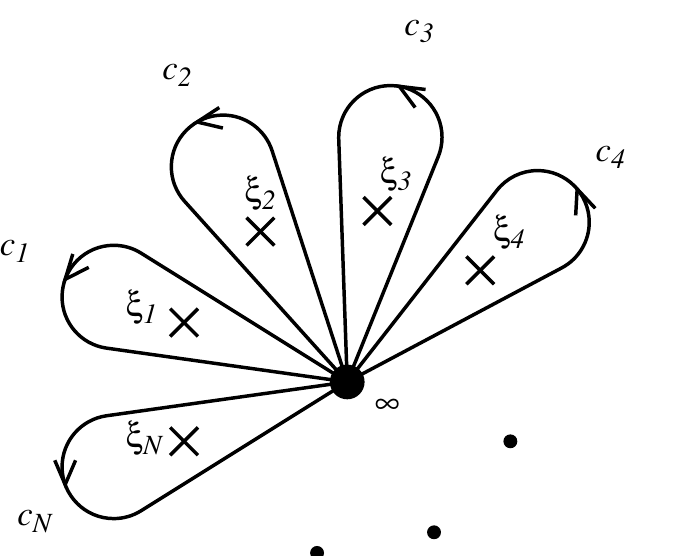}
\end{center}
\caption{The generators $c_1,\ldots,c_N$ of $\pi_1(Y)$ are loops around each of the ramification points.  Note that the product $c_1\cdots c_N=1$, so it suffices to use $c_1,\ldots,c_{N-1}$ as generators.}
\label{fig:loops-c}
\end{figure}

It will be convenient to use a different set of loops, $w_1,\ldots, w_{N-1}$, defined so that $w_j$ is the simple closed curve starting and ending at $\infty$, that goes around $\{\xi_1,\ldots,\xi_j\}$.  See Figure~\ref{fig:loops-w}.  More symbolically,
\begin{align*}
w_1&=c_1\\
w_2&=c_1 c_2\\
\vdots&\makebox[2em]{}\vdots\\
w_{N-1}&=c_1c_2\cdots c_{N-1}
\end{align*}

Likewise, we can write
\[c_i=\begin{cases}
w_1&i=1\\
w_{i-1}{}^{-1}w_i&i=2,\ldots,N-1\\
w_{N-1}{}^{-1}&i=N
\end{cases}
\]
and thus, $w_1,\ldots,w_{N-1}$ is also a set of generators for $\pi_1(Y)$.

\begin{figure}
\begin{center}
\includegraphics[width=3.5in]{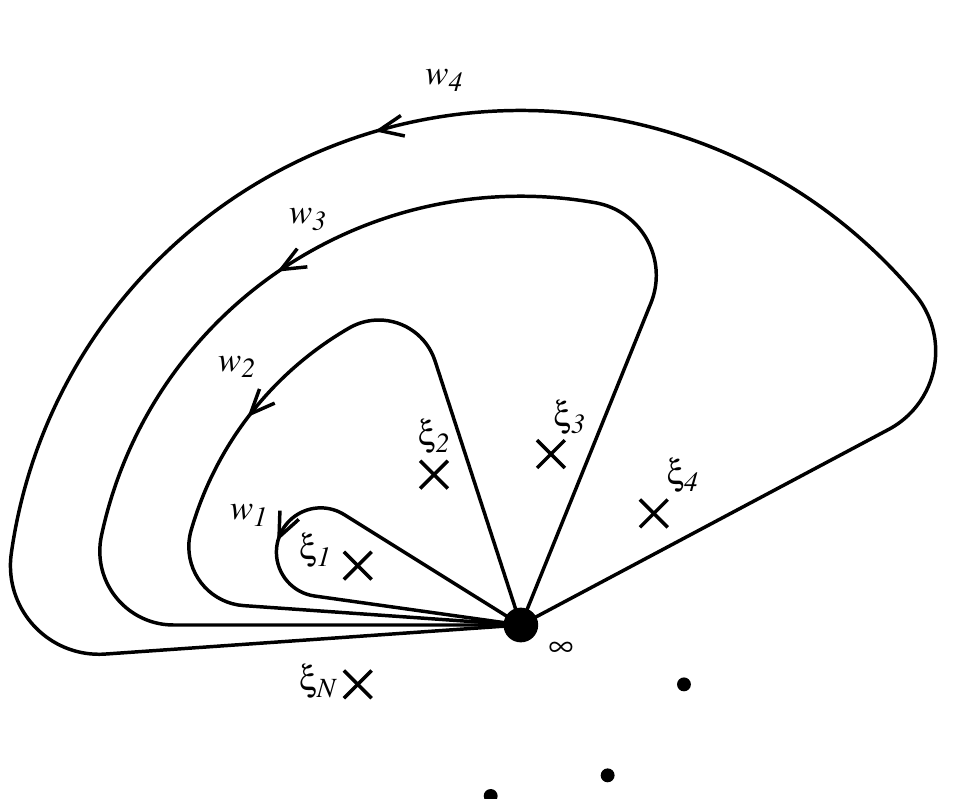}
\end{center}
\caption{The generators $w_1,\ldots,w_{N-1}$ of $\pi_1(Y)$ are an alternative set of generators of the monodromy group.}
\label{fig:loops-w}
\end{figure}

Up to homotopy, the loop $w_j$ can be represented by the concatenation of coloured edges in the beachball $\CP^1$: more specifically, the path $z_{j+1}(t)$ from $\infty$ to $0$, followed by $z_1(t)$ from $0$ to $\infty$.

To find the monodromy $\sigma_j$ associated with $w_j$, pick any fermion $x_0\in F = \beta_X^{-1}(\infty)$.  The path $w_j$ has a lift $\tilde{w}_j$ in $X$ that starts at $x_0$.  This lifted path again follows coloured edges but these are now in the Adinkra.  Again, $\tilde{w}_j$ follows an edge of colour $j+1$, then an edge of colour $1$.  Since $w_j$ started and ended at $\infty$, it must be that $\tilde{w}_j$ ends at a fermion.  Then define $\sigma_j(x_0)$ to be the fermion at the endpoint of $\tilde{w}_j$.  In this way, $\sigma_j$ is a map from fermions to fermions, and is thus a permutation on the set of fermions.  The $\sigma_j$ generate the monodromy group, denoted $\mathcal{M}$.

This was analyzed in Ref.~\cite{geom1}, and for a connected Adinkra,
\[\mathcal{M}\cong \F_2^{N-k-1}.\]
More precisely, a connected Adinkra is associated to a doubly even code $C$ of length $N$ of dimension $k$ \cite{at}.
 This is a $k$-dimensional vector subspace of $\F_2^N$.  
 Define \[E=\{(x_1,\ldots,x_N)\,|\,\sum x_i\equiv 0 \pmod{2}\},\]
which is also a vector subspace of $\F_2^N$ of dimension $N-1$, and note that $C\subset E\subset \F_2^N$.

For an $N$-dimensional cubical Adinkra, with vertex set $\{0,1\}^N$, moving along colour $j$ means translating modulo $2$ by the standard basis vector with a single $1$ in the $j$th coordinate.  Then $\sigma_{w_j}$ is translation modulo $2$ by
\[(1,0,\ldots,0,1,0,\ldots,0)\]
with a $1$ in the first and $j+1$st coordinate.  In $\F_2^N$, these generate $E$.  Then $\mathcal{M}\cong E\cong \F_2^{N-1}$.

For a more general connected Adinkra, which is a quotient of the cubical Adinkra by a code $C$, the monodromy group is obtained by quotienting $E$ by the code $C$, so that $\mathcal{M}\cong E/C\cong \F_2^{N-k-1}$.  
We can write the following short exact sequence:
\begin{equation}
\xymatrix{
1\ar[r]&C\ar[r]&E\ar[r]&\mathcal{M}\ar[r]&1.
}
\label{eqn:monodromy}
\end{equation}
A summary of short exact sequences with examples pertinent to this paper is given in Appendix~\ref{a:SES}.

\section{Signed permutations}\label{sec:signed_perms}
If $S$ is a set, then a permutation on $S$ is a bijection from $S$ to itself.  In this section, we will define a signed permutation using a similar idea, where the objects of $S$ have signs.  To do this, we must first define the notion of signed set.


A {\em signed set} is a set $S$, together with a free $\Z/2\Z$ action on the set.  Since there will be several distinct r\^{o}les of the group $\Z/2\Z$ in this paper, we will write this group as $\{\pm 1\}=\{1,-1\}$ when in this context.  If $x\in S$, we write $-x$ for $(-1)x$.  The orbits of $S$ partition $S$ into subsets, each of which have precisely two elements; let $|S|$ denote the set of such orbits.  

If $S$ and $T$ are signed sets, then a {\em signed set morphism} from $S$ to $T$ is a function $f:S\to T$ so that $f(-x)=-f(x)$ for all $x\in S$.  This corresponds to the standard notion of a morphism of sets on which a group acts.  If this map is bijective, we say it is a signed set isomorphism.  A signed set morphism $f:S\to T$ gives rise to a set map $\abs(f):|S|\to |T|$ that sends $\{x,-x\}$ to $\{f(x),f(-x)\}$.  This is functorial in the sense that if $f:S\to T$ and $g:T\to U$ are signed set morphisms, then $\abs(g\circ f) = \abs(g) \circ \abs(f):|S|\to |U|$.

A {\em signed permutation} on $S$ is a signed set isomorphism from $S$ to itself.  The set of signed permutations on $S$ is a group under composition, and is called the signed permutation group, $BC(S)$.  If $S=\{\pm 1,\ldots,\pm n\}$, then we write $BC_n=BC(S)$.\footnote{There is no universally recognized standard notation for this group in the literature.  We follow the notation that arises from the classification of Coxeter groups, where the groups called $B_n$ and those called $C_n$ coincide, and are thus sometimes called $BC_n$.  It is also sometimes called the hyperoctahedral group, because it is the group of symmetries for the hyperoctahedron, which is the convex hull of the vectors $\{\pm \vec{e}_1, \ldots, \pm \vec{e}_n\}\in\R^n$.}  The map
\[\abs:BC_n\to S_n\]
that takes a signed permutation $f$ to the permutation $|f|$ on the $n$-element set $|S|$ is a homomorphism.

\subsection{Signed permutation matrices}
  Recall that an $n\times n$ permutation matrix is an $n\times n$ matrix where every row and every column has exactly one nonzero entry, and where the nonzero entries must be $1$. 
  Such a matrix corresponds to the linear automorphism of $\R^n$ induced by a permutation of the $n$ standard basis vectors $\{\vec{e}_1,\ldots,\vec{e}_n\}$.
Conversely, given a set $S$, and an ordering on $S=\{x_1,\ldots,x_n\}$, a permutation $\lambda$ on $S$ gives rise to an $n\times n$ matrix $L$, whose $i$th row and $j$th column is
\[L_{i,j}=\begin{cases}
1,&\mbox{if $\lambda(x_j)=x_i$,}\\
0,&\mbox{otherwise}
\end{cases}\]

Likewise, an $n\times n$ signed permutation matrix is an $n\times n$ matrix where every row and every column has exactly one nonzero entry, and where the nonzero entries must be either $1$ or $-1$.  If we take the corresponding linear automorphism of $\R^n$, and restrict to the set $\{\vec{e}_1,\ldots,\vec{e}_n,-\vec{e}_1,\ldots,-\vec{e}_n\}$, viewed as a signed set in the obvious way, we obtain a signed permutation.

And conversely, given a signed set $S$, suppose we have an ordered subset $T=\{x_1,\ldots,x_n\}$ of $S$ so that for each $x\in S$, exactly one of $x$ or $-x$ is in $T$.  Then a signed permutation $\lambda$ on $S$ gives rise to the $n\times n$ matrix $L$ whose $i$th row and $j$th column is
\[L_{i,j}=\begin{cases}
1,&\mbox{if $\lambda(x_j)=x_i$,}\\
-1,&\mbox{if $\lambda(x_j)=-x_i$,}\\
0,&\mbox{otherwise.}
\end{cases}
\]

If $A$ is a signed permutation matrix, we let $|A|$  be the matrix obtained by taking the absolute value of each entry.  Then $|A|$ is a permutation matrix and this notation is compatible with the definition of $\abs(f)$ for signed permutations $f$.  Let $D_A$ be the diagonal signed permutation matrix whose $i$-th diagonal entry is $1$ or $-1$ according to the sign of the nonzero entry in the $i$-th row of $A$.

If $P$ is a permutation matrix and $D$ is a diagonal signed permutation matrix, then $PDP^{-1}$ is a diagonal signed permutation matrix.  Viewing the set of diagonal signed permutation matrices as $\{\pm 1\}^n$, then this forms a normal subgroup of $BC_n$, and its quotient is the symmetric group $S_n$.  
Thus, we have the following short exact sequence:
\begin{equation}\label{e:SESBCn}
\xymatrix{
1\ar[r]&\{\pm 1\}^n\ar[r]&BC_n\ar[r]^{\abs}&S_n\ar[r]&1
}
\end{equation}
Since any permutation matrix can be viewed as a signed permutation matrix, the map $\textrm{abs}$ admits a section and the exact sequence is split. 
Thus, $BC_n$ is a semidirect product:
\[BC_n = \{\pm 1\}^n \rtimes S_n.\]
Thus, each signed permutation matrix $A$ admits a unique factorization:
\begin{equation}
A=B\,P
\end{equation}
where $B$ is a diagonal signed permutation matrix, and $P$ is a permutation matrix.  The unique solution is $B=D_A$ and $P=|A|$.
This is a manifestation of the fact that a signed permutation can be viewed as the composition of a permutation on the pairs of the form $\{x,-x\}$, followed by a specification of whether an element $x$ goes to $y$ or $-y$.

\section{The Signed Monodromy group}\label{sec:SMG}
Given an Adinkra and a corresponding Riemann surface $X$ with $\beta_N:X\to \CP^1$, as defined in Section~\ref{sec:belyi}, the dashing of the edges of the Adinkra describes a signed monodromy group.  
Specifically, take $\infty$ as our basepoint and let $F=\beta_N^{-1}(\infty)$ be the set of fermions; let $-F$ be another disjoint copy of $F$.  
Let $\pm F$ be the union $F\cup -F$, defined as the signed set which pairs elements of $F$ with the corresponding element of $-F$.

We again take the loops $w_j$ in $Y=\CP^1-\{\xi_1,\ldots,\xi_N\}$ based at $\infty$, as described above, and again apply a homotopy to bring these loops to the composition of coloured paths $z_{j+1}$ and $z_1$.  For each $w_j$, we define a signed permutation $\zeta_j$ as follows: if $x_0\in F$, then the path $w_j$ has a lift $\tilde{w}_j$ that starts at $x_0$.  This lifted path again follows coloured edges but these are now in the Adinkra.  Again, $\tilde{w}_j$ follows an edge of colour $j+1$, then an edge of colour $1$, and ends at $\sigma_j(x_0)$.

Now define $\zeta_j(x_0)$ to be $(-1)^s\sigma_j(x_0)$ where $s$ is the number of dashed edges in $\tilde{w}_j$, and $\sigma_j(x_0)$ as before is the endpoint of $\tilde{w}_j$.  Likewise,
$\zeta_j(-x_0) = (-1)^{s+1}\sigma_j(x_0)$.  
In this way, $\zeta_j$ is a signed permutation of the signed set $\pm F$.  
\begin{definition}
	Notations as above, the group generated by the $\zeta_1,\ldots,\zeta_N$ is called the {\em signed monodromy group}, which we denote by $\mathcal{H}$.
\end{definition}

Taking the absolute value of these signed monodromies gives  the old monodromy group $\mathcal{M}$ from Section~\ref{sec:belyi}:

\begin{proposition}
There is an epimorphism
\[\abs:\mathcal{H}\to\mathcal{M}.\]
\end{proposition}
\begin{proof}
We identify $F$ with $|\pm F|$, and let $\textrm{abs}$ denote the homomorphism defined in Section~\ref{sec:signed_perms}.
To show $\abs$ is onto, suppose $m\in\mathcal{M}$.
Then there is a loop in $Y$ that gives rise to $m$.
This loop defines a signed monodromy $h$ so that $\abs(h)=m$.
\end{proof}

If we define $\Sigma$ to be the kernel of the above map, then we have the short exact sequence:
\begin{equation}
\xymatrix{
1\ar[r]&\Sigma\ar[r]&\mathcal{H}\ar[r]^\abs&\mathcal{M}\ar[r]&1
}
\label{eqn:HM}
\end{equation}

\section{Salingaros Vee groups}\label{sec:vee}
In this section, we review the Vee groups due to Salingaros\cite{salingaros1}.

\begin{definition}
The Salingaros Vee Group, denoted by $\G_n$, is the group with the following presentation:
\[\G_n=\gen{-1,g_1,\ldots,g_n\,|\,
(-1)^2=1, g_i{}^2=-1, g_ig_j=-1g_jg_i, (-1)g_i = g_i(-1)}.\]
\end{definition}
It has a central element $-1$ of order two, and $n$ other generators $g_1,\ldots,g_n$ which square to $-1$ and anticommute with each other.
These groups relate to the Clifford algebras $\Cl_n$, or more generally $\Cl_{p,q}$, in the same way that the quaternionic group $Q_8$ relates to the quaternion algebra: each is a finite multiplicative subgroup that contains the defining generators. 
For more basic information about Clifford algebras, see Ref.~\cite{rLM}.

When $n$ is even, $\G_n$ is an example of an extra special $2$-group.

This group can be viewed as a signed set, with the natural meaning of $-1$.
 Since the element $-1$ is central, the action of $\G_n$ on itself by left multiplication is a signed permutation.
Given a concatenated string of generators in $\G_n$, it is straightforward to use the anticommutation relations to arrange the $g_i$ in ascending order, and then use the fact that $g_i{}^2=-1$ to insist that each $g_i$ occurs at most once.  Then, the fact that $-1$ is central and squares to $1$ can be used to demonstrate the following:
\begin{proposition}
Every element of $\G_n$ can be written uniquely as
\[(-1)^bg_1^{x_1}\cdots g_n^{x_n}\]
for $b, x_1,\ldots, x_n \in \{0,1\}$.  Therefore $\G_n$ has $2^{n+1}$ elements.
\end{proposition}
Uniqueness can be shown by comparing two such strings and cancelling out common factors of $g_i$, until we can write one of the generators in terms of the others.

Alternately, this group can be constructed explicitly from formal strings of the form given in the Proposition, and defining the multiplication so that the anticommutation relations hold.


\begin{remark}
In Ref.~\cite{Simon-rep}, there is a similar construction, except that the generators square to $1$ instead of $-1$.  The same techniques work, and the main results are analogous.  More generally Ref.~\cite{ablamowicz} defines the groups $G_{p,q}$, where $p$ generators square to $1$, and $q$ generators square to $-1$.
%
%
%
\end{remark}

\begin{example}
$\G_0=\{\pm 1\}\cong\Z/2\Z$.
\end{example}

\begin{example}
$\G_1\cong \Z/4\Z$.  This is because $g_1$ squares to $-1$, and so $-1$ is not needed as a generator.  So $g_1$ generates the group and is of order 4.
\end{example}

\begin{example}
$\G_2\cong Q_8$.  The group $Q_8$ is the famous ``unit quaternion'' group of order $8$,
\[Q_8=\{\pm 1,\pm i, \pm j, \pm k\}\]
The isomorphism here sends $-1$ to $-1$, and $g_1$ to $i$ and $g_2$ to $j$.  Then $g_1g_2$ goes to $k$.  The relations that define $\G_2$ are consistent with the relations that define $Q_8$.
\end{example}

\begin{example}
$\G_3\cong Q_8\times \Z/2\Z$.  The isomorphism sends $g_1$ to $(i,0)$, $g_2$ to $(j,0)$, and $g_3$ to $(k,1)$.
\end{example}

Other examples result in groups that are less familiar.

\section{Main theorem}\label{sec:main}
Let $h_1=111\cdots 1 \in \F_2^N$ be the word of all ones.  Since the code $C$ is doubly even, $h_1$ cannot be in $C$ unless $N$ is a multiple of $4$.

The main theorem is the following:

\begin{theorem}
Suppose a connected Adinkra is given, and has code $C$.  The groups $\mathcal{H}$ and $\Sigma$ depend on whether $h_1\in C$ or not:
\begin{itemize}
\item If $h_1\not\in C$, then  $\mathcal{H}\cong \G_{N-1}$ and $\Sigma\cong \F_2^{k+1}$.
\item If $h_1\in C$, then  $\mathcal{H}\cong \G_{N-2}$ and $\Sigma\cong\F_2^{k}$
\end{itemize}
\label{thm:main2}
\end{theorem}

\begin{example}[$k=0$]  In this specific case, the code $C$ is trivial, i.e., $C=\{000\cdots 0\}$.  Then $h_1\not\in C$, $\mathcal{H}\cong \G_{N-1}$, and $\Sigma=\{\pm 1\}$.  The quotient of $\mathcal{H}$ by $\{\pm 1\}$ gives rise to the monodromy group $\mathcal{M}\cong E\cong \F_2^{N-1}$.

For each such unsigned monodromy there are two signed monodromies, which are negatives of each other.  Suppose we follow a loop consisting of a concatenation of various $w_j$s.  Whereas the order of the $w_j$s does not affect the unsigned monodromy, it can affect the sign in the signed monodromy.
\end{example}

\begin{example}[$N=4,\ C=\gen{(1111)}$]
In this case, $h_1=1111\in C$, and according to the theorem, $\mathcal{H}\cong \G_2\cong Q_8$, and $\Sigma=\{1,-1\}\cong \F_2$.  There are $2^2=4$ bosons and 4 fermions.  The unsigned monodromy group $\mathcal{M}\cong \F_2^2$ acts on these fermions freely, transitively, and faithfully.  To each unsigned monodromy, there are two signed monodromies, each the negative of the other.
\end{example}

\begin{example}[$N=5,\ C=\gen{(11110)}$]
In this case, note that $h_1=11111$ is not in the code (it cannot be since it has weight 5).  According to the theorem, $\mathcal{H}\cong \G_4$, $\Sigma\cong \F_2^2$, and $\mathcal{M}\cong \F_2^3$.  There are 8 bosons and 8 fermions.  The 8 unsigned monodromies act on these fermions freely, transitively, and faithfully.  To each unsigned monodromy, there are four signed monodromies.
\end{example}


\section{Signed Monodromies and the $\GR(d,N)$ algebra}
\label{sec:proof1}
Given an Adinkra, let $B=\{b_1,\ldots,b_d\}$ be the set of bosons and let $F=\{f_1,\ldots,f_d\}$ be the set of fermions. 
 Define formal negatives $-B=\{-b_1,\ldots,-b_d\}$ and $-F=\{-f_1,\ldots,-f_d\}$.
 The sets $\pm B = B\cup -B$ and $\pm F=F\cup -F$ are then signed sets.

We define signed set homomorphisms $\lambda_1,\ldots,\lambda_N$ from $\pm F$ to $\pm B$, as follows.  
If $f_j$ is a fermion, then $\lambda_i(f_j)=b_k$ if there is a solid edge of colour $i$ from $f_j$ to $b_k$, and $\lambda_i(f_j)=-b_k$ if there is a dashed edge of colour $i$ from $f_j$ to $b_k$.  Because each vertex has exactly one edge incident with it of each colour, this is well-defined.  
Likewise, define $\rho_1,\ldots,\rho_N$ from $\pm B$ to $\pm F$ where $\rho_i(b_j)=f_k$ if there is a solid edge of colour $i$ from $b_j$ to $f_k$ and $\rho_i(b_j)=-f_k$ if there is a dashed edge of colour $i$ from $b_j$ to $f_k$.  

Then the signed monodromy $\zeta_j$ is then
\[\zeta_j=\rho_1\lambda_{j+1}.\]

Since an edge from the boson $b_j$ to the fermion $f_k$ is also an edge from the fermion $f_k$ to the boson $f_j$, we see that
\begin{equation}
\rho_i=\lambda_i{}^{-1},
\label{eqn:susy1}
\end{equation}
and, in particular, the $\rho_i$ and $\lambda_i$ are signed set isomorphisms.

By the odd dashing property Adinkras, if $i\not=j$, then
\begin{align}
\lambda_i\rho_j&=-\lambda_j\rho_i\label{eqn:susy2}\\
\rho_i\lambda_j&=-\rho_j\lambda_i
\label{eqn:susy3}
\end{align}

As a consequence, we have the following:
\begin{lemma}
For all $j$,
\[\zeta_j{}^2=-1\]
and for all $i\not=j$,
\[\zeta_i\zeta_j=-\zeta_j\zeta_i.\]
\label{lem:signedholorelations}
\end{lemma}
\begin{proof}
\begin{align*}
\zeta_j{}^2&=\rho_1(\lambda_{j+1}\rho_1)\lambda_{j+1}\\
&=-\rho_1\lambda_1\rho_{j+1}\lambda_{j+1}\\
&=-1
\end{align*}
\begin{align*}
\zeta_i\zeta_j&=\rho_1\lambda_{i+1}\rho_1\lambda_{j+1}\\
&=-\rho_1\lambda_1\rho_{i+1}\lambda_{j+1}\\
&=-\rho_{i+1}\lambda_{j+1}.
\end{align*}
Switching $i$ and $j$ shows that $\zeta_j\zeta_i = -\rho_{j+1}\lambda_{i+1}$, which is $\rho_{i+1}\lambda_{j+1}=-\zeta_i\zeta_j$.
\end{proof}

\begin{corollary}
If $N\ge 1$, then $-1\in\mathcal{H}$, and if $\sigma\in \mathcal{H}$, then $-\sigma\in\mathcal{H}$ also.
\end{corollary}

If we number the bosons $b_1,\ldots,b_d$ and number the fermions $f_1,\ldots,f_d$, then these signed isomorphisms can be written as matrices.  If we write $L_i$ for the signed permutation matrix corresponding to $\lambda_i$ and $R_i$ for the signed permutation matrix corresponding to $\rho_i$, then (\ref{eqn:susy1}, (\ref{eqn:susy2}), and (\ref{eqn:susy3}) can be phrased as
\begin{align}
L_iR_j+L_jR_i&=2\delta_{ij}\id\label{eqn:susy3a}\\
R_iL_j+R_jL_i&=2\delta_{ij}\id
\label{eqn:susy3b}
\end{align}
which is the form found in Ref.~\cite{enuf,rGR1,rGR2} and is known as the algebra of general, real $d\times d$ matrices describes $N$ supersymmetries: the $\GR(d,N)$ algebra.

\section{Properties of the Salingaros Vee groups}
In order to prove the main theorem, and in order to best use the results, it will be necessary to first prove some basic facts about Salingaros Vee groups $\G_n$.  The results described here are found in Refs.~\cite{salingaros1, Simon-rep, ablamowicz}, but we restate them here for completeness.

\begin{proposition}
There is a group epimorphism
\[\abs:\G_n \to \F_2^n\]
with kernel $\{1,-1\}$.  In other words, the following is a short exact sequence of groups:
\begin{equation}
\xymatrix{1\ar[r]&\{\pm 1\} \ar[r]&\G_n \ar[r]^{\abs}&\F_2^n \ar[r]&1}
\label{eqn:fclexact}
\end{equation}
\label{prop:fclabs}
\end{proposition}
\begin{proof}
If we quotient by $\{1,-1\}$, then the resulting relations say that the generators commute and are of order $2$. 
Thus,  the quotient is isomorphic to $\F_2^n$.
\end{proof}

\begin{proposition}
If 
\[x=(-1)^bg_1^{x_1}\cdots g_n^{x_n}\]
and $g_j$ is a generator then
\[g_jx=xg_j\]
if $\sum_{i\not=j} x_i$ is even.  Otherwise,
\[g_jx=-xg_j.\]
\label{prop:commute}
\end{proposition}
\begin{proof}
This is a tedious but straightforward calculation that follows from the anticommutativity of the $g_i$.
\end{proof}

\begin{proposition}
Let $x\in \G_n$.  The conjugacy class of $x$ is either $\{x\}$, if $x$ is in the centre, or $\{x,-x\}$ if it is not.
\label{prop:conj}
\end{proposition}
\begin{proof}
The results from Proposition~\ref{prop:commute} imply that $g_jxg_j^{-1}$ is either $x$ or $-x$.  By induction, a conjugate of $x$ can only be $x$ or $-x$.  Trivially, the conjugacy class of $x$ must contain $x$.  The statement that it contains only $x$ is equivalent to the statement that $x$ is in the centre of $\G_n$.
\end{proof}

\begin{definition}
In $\G_n$, define $\omega=g_1\cdots g_n$.
\end{definition}

\begin{proposition}
The centre of $\G_n$ is
\[Z(\G_n)=\begin{cases}
\{\pm 1\},&\mbox{$n$ is even}\\
\{\pm 1,\pm \omega\},&\mbox{$n$ is odd}
\end{cases}\]
\label{prop:centre}
\end{proposition}
\begin{proof}
Let $x$ be in the centre.  Then it commutes with each of the $g_j$.  Write
\[x=(-1)^bg_1^{x_1}\cdots g_n^{x_n}.\]
By Proposition~\ref{prop:commute}, $x_1=\cdots=x_n$.  If these are all $0$, $x=\pm 1$.  If these are all $1$, then $x=\pm \omega.$  Finally, $\pm\omega$ commutes with all $g_j$ if and only if $n$ is odd.
\end{proof}

\begin{proposition}
\[\omega^2=(-\omega)^2=\begin{cases}
1,&\mbox{if $n\equiv 0, 3\pmod{4}$}\\
-1,&\mbox{if $n\equiv 1, 2\pmod{4}$}
\end{cases}
\]
\label{prop:omegasquared}
\end{proposition}
\begin{proof}
In rewriting $\omega^2=g_1\cdots g_n g_1\cdots g_n$ in order, we do ${n\choose 2}=n(n-1)/2$ swaps.  This is even if and only if $n$ is congruent to $0$ or $1$ modulo $4$.  The result when this is done is
\[g_1{}^2\cdots g_n{}^2=(-1)^n.\]
This is $1$ if $n$ is even, and $-1$ if $n$ is odd.
\end{proof}



\subsection{Normal subgroups of $\G_n$}
We now classify all normal subgroups of $\G_n$.  As we will see, there are two kinds: those of the type $\abs^{-1}(V)$ for some subgroup $V$ of $\F_2^n$, and those that are contained in the centre.

First, the normal subgroups of the first type.
\begin{proposition}
Given a subgroup $V$ of $\F_2^n$, $\abs^{-1}(V)$ is a normal subgroup of $\G_n$ that contains $-1$.
\label{prop:preimage}
\end{proposition}
\begin{proof}
Since $\F_2^n$ is abelian, any subgroup of it is automatically normal, and the preimage of a normal subgroup under the group homomorphism $\abs$ is a normal subgroup of $\G_n$.  Since the kernel of $\abs$ is $\{1,-1\}$, we have that $-1$ is in such a preimage.
\end{proof}

The main observation is the following:
\begin{proposition}
Every normal subgroup of $\G_n$ either contains $-1$ or is contained in the centre.
\label{prop:normal}
\end{proposition}
\begin{proof}
Suppose $G$ is a normal subgroup of $\G_n$ that is not contained in the centre.  Let $x\in G$ be not in the centre of $\G_n$.  Since $G$ is normal, all conjugates of $x$ are in $G$.  Then by Proposition~\ref{prop:conj}, $-x\in G$.  Since $x$ and $-x$ are in $G$, we conclude that $-1\in G$.
\end{proof}

Note that of course, it is possible for a normal subgroup to be both contained in the centre and contain $-1$.

These facts are what we need to prove:
\begin{theorem}
Let $G$ be a normal subgroup of $\G_n$.
\begin{itemize}
\item If $-1\in G$, then $G=\abs^{-1}(V)$ for some $V\in \F_2^n$.
\item If $-1\not\in G$, then either $G=\{1\}$, or $n\equiv 3\pmod{4}$ and  either $G=\{1,\omega\}$ or $G=\{1,-\omega\}$.
\end{itemize}
\label{thm:normal}
\end{theorem}

\begin{proof}
Let $G$ be a normal subgroup of $\G_n$ that contains $-1$.  Then, by closure, for every $x\in G$, $-x=(-1)x\in G$.  Therefore $G=\abs^{-1}(V)$ for some subset $V$ of $\F_2^n$.  Since $\abs$ is onto, we have that $V=\abs(\abs^{-1}(V))$.  This is $\abs(G)$, which is a subgroup of $\F_2^n$.

Now suppose $G$ is a normal subgroup of $\G_n$ that does not contain $-1$.  By Proposition~\ref{prop:normal}, we have that $G$ is contained in the centre $Z(\G_n)$.  By Proposition~\ref{prop:centre}, this means $G$ is trivial if $n$ is even.  If $n\equiv 1\pmod{4}$, by Proposition~\ref{prop:omegasquared}, $G$ containing $\omega$ implies $G$ contains $-1$.  Likewise $(-\omega)^2=-1$ and $G$ containing $-\omega$ implies $G$ contains $-1$.  Therefore, if $G$ does not contain $-1$, then $G$ must be trivial.  For $n\equiv 3\pmod{4}$, we simply examine the subgroups of the centre $\{1,-1,\omega,-\omega\}$ that do not contain $-1$.
\end{proof}

\begin{proposition}
If $n\equiv 3\pmod{4}$, then
\[\G_n\cong \G_{n-1}\times \Z/2\Z\]
We also have
\[\G_n/\{1,\omega\}\cong \G_{n-1}\cong \G_n/\{1,-\omega\}.\]
\label{prop:quotientomega}
\end{proposition}
\begin{proof}
First, if $n\equiv 3\pmod{4}$, then $\omega^2=1$ by Proposition~\ref{prop:omegasquared}.  Thus $\{1,\omega\}$ is a normal subgroup of $\G_n$ and is isomorphic to $\Z/2\Z$.

Likewise, the subgroup $J$ generated by $\{g_1,\ldots,g_{n-1}\}$ is isomorphic to $\G_{n-1}$ and is normal, by Proposition~\ref{prop:conj}.

Since $\omega\not\in J$, we have that $\{1,\omega\}\cap J=\{1\}$.  Therefore these two subgroups form $\G_n$ as an internal direct product.  The quotient $\G_n/\{1,\omega\}$ is therefore isomorphic to $\G_{n-1}$.

The same arguments work for $\{1,-\omega\}$.
\end{proof}

\section{Proof of main theorem}
\label{sec:proof2}
In this section we will prove the main theorem, Theorem~\ref{thm:main2}, computing $\mathcal{H}$ and $\Sigma$.

The $\zeta_j$ satisfy the same relations in Lemma~\ref{lem:signedholorelations} as the generators of $\G_{N-1}$.  This proves:
\begin{proposition}
There is a group epimorphism
\[\phi:\G_{N-1}\to\mathcal{H}.\]
with $\phi(g_j)=\zeta_j$ and $\phi(-1)=-1$.
\end{proposition}
\begin{proof}
Define $\phi(g_j)=\zeta_j$ and $\phi(-1)=-1$, and extend $\phi$ to products of these generators in order to ensure that $\phi$ is a homomorphism.  Lemma~\ref{lem:signedholorelations} guarantees that this is well-defined.  The $\zeta_j$ generate $\mathcal{H}$, so that this map is onto.
\end{proof}

If we define $K=\ker(\phi)$, then $\mathcal{H}\cong \G_{N-1}/K$.  This results in the following short exact sequence:

\begin{equation}
\xymatrix{
1\ar[r]&K \ar[r]&\G_{N-1} \ar[r]^{\phi}&\mathcal{H} \ar[r]&1\\
}
\label{eqn:kfclh}
\end{equation}

We can use this, and the short exact sequences (\ref{eqn:monodromy}), (\ref{eqn:HM}), and (\ref{eqn:fclexact}), to put together the following diagram:

\begin{equation}
\xymatrix{
&&1\ar[d]&1\ar[d]\\
&&\{\pm 1\}\ar[d]&\Sigma\ar[d]\\
1\ar[r]&K\ar[r]&\G_{N-1}\ar[r]^{\phi}\ar[d]^{\abs}&\mathcal{H}\ar[r]\ar[d]^{\abs}&1\\
1\ar[r]&C \ar[r]&E\cong\F_2^{N-1} \ar[r]^{\rho}\ar[d]&\mathcal{M} \ar[r]\ar[d]&1\\
&&1&1&
}
\label{eqn:diagram1}
\end{equation}

Here we have defined $\Sigma$ to be the kernel of $\abs:\mathcal{H}\to\mathcal{M}$.  The isomorphism $E\cong \F_2^{N-1}$ is obtained by taking $\pi:\F_2^N\to \F_2^{N-1}$, the projection onto the last $N-1$ coordinates, and restricting to $E$.  The result is a linear map $\pi|_E:E\to\F_2^{N-1}$.  Since the kernel of this is trivial, and the dimensions of $E$ and $\F_2^{N-1}$ are both $N-1$, $\pi|_E$ is an isomorphism.  The inverse takes $(y_1,\ldots,y_{N-1})\in \F_2^{N-1}$ to $(y_0,y_1,\ldots,y_{N-1})\in E$, where $y_0=\sum_{i=1}^{N-1} y_i\pmod{2}$.

\begin{proposition}
The diagram (\ref{eqn:diagram1}) above is commutative.
\end{proposition}
\begin{proof}
Let
\[y=(-1)^sg_1^{y_1}\cdots g_{N-1}^{y_{N-1}}\in \G_{N-1}.\]
Then
\[\phi(y)=(-1)^s\zeta_1^{y_1}\cdots \zeta_{N-1}^{y_{N-1}}.\]
If $x_0$ is any fermion, then
\[\phi(y)(x_0)=\pm \sigma_1^{y_1}\cdots\sigma_{N-1}^{y_{N-1}}(x_0)\]
and
\[\abs(\phi(x))(|x_0|)=\sigma_1^{y_1}\cdots\sigma_{N-1}^{y_{N-1}}(|x_0|).\]
Likewise,
\[\abs(y)=(y_1,\ldots,y_{N-1})\]
and
\[\rho(\abs(y))(|x_0|)=\sigma_1^{y_1}\cdots \sigma_{N-1}^{y_{N-1}}(|x_0|).\]
\end{proof}

\begin{proposition}
The diagram (\ref{eqn:diagram1}) can be extended to the following commutative diagram:
\begin{equation}
\xymatrix{
&&1\ar[d]\ar[r]&1\ar[d]\\
&1\ar[r]\ar[d]&\{\pm 1\}\ar[r]^{\lambda}\ar[d]&\Sigma\ar[d]\\
1\ar[r]\ar[d]&K\ar[r]\ar[d]^{\abs}&\G_{N-1}\ar[r]^{\phi}\ar[d]^{\abs}&\mathcal{H}\ar[r]\ar[d]^{\abs}&1\ar[d]\\
1\ar[r]&C \ar[r]&\F_2^{N-1} \ar[r]^{\rho}\ar[d]&\mathcal{M} \ar[r]\ar[d]&1\\
&&1\ar[r]&1&
}
\label{eqn:diagram}
\end{equation}
\end{proposition}
\begin{proof}
Note that the unlabeled maps $K\to \G_{N-1}$, $C\to \F_2^{N-1}$, $\{\pm 1\}\to \G_{N-1}$, and $\Sigma\to \mathcal{H}$ are all inclusion maps.

We first demonstrate the existence of the $\abs$ map from $K$ to $C$.  We first restrict $\abs$ to $K=\ker(\phi)$.  We need to show $\abs(K)\subset C$.  Then the map we want is simply the restriction of $\abs$ to $K$.

Let $k\in K$.  Then
\[\rho(\abs(k))=\abs(\phi(k))=\abs(1)=1\]
so $\abs(k)\in \ker(\rho)=C$.

The fact that $\abs:K\to C$ is the restriction of $\abs:\G_{N-1}\to \F_2^{N-1}$ shows that the following part of the diagram is commutative.
\[
\xymatrix{
K\ar[r]\ar[d]^{\abs}&\G_{N-1}\ar[d]^{\abs}\\
C \ar[r]&\F_2^{N-1}
}
\]

An analogous argument shows that $\lambda:\{\pm 1\}\to \Sigma$ exists and is the restriction of $\phi$ to $\{\pm 1\}$.  In turn, the fact that $\lambda$ is the restriction of $\phi$ shows that the following part of the diagram is commutative:
\[
\xymatrix{
\{\pm 1\}\ar[r]^{\lambda}\ar[d]&\Sigma\ar[d]\\
\G_{N-1}\ar[r]^{\phi}&\mathcal{H}
}
\]

To show that $\abs:K\to C$ is a monomorphism, suppose $k\in K$ so that $\abs(k)=1$.  Then in $\G_{N-1}$ we know that $k=1$ or $k=-1$.  But $-1\not\in K$, because $\phi(-1)=-1$.  Therefore $k=1$.

The fact that $\lambda$ is a monomorphism follows from the fact that $\phi(-1)=-1$.

The maps in the diagram involving the trivial group $1$ are the trivial maps.  Commutativity of the squares that involve $1$ follows.
\end{proof}

\subsection{Calculating $K$}
\begin{proposition}
If $N\equiv 0\pmod{4}$, then $K$ can be $\{1\}$, $\{1,\omega\}$, or $\{1,-\omega\}$.  Otherwise, $K=\{1\}$.
\label{prop:k0}
\end{proposition}
\begin{proof}
We note that $K$ is a normal subgroup of $\G_{N-1}$, and since $\phi(-1)=-1$, we see that $-1\not\in K$.  The result is a consequence of Theorem~\ref{thm:normal}, applied to $n=N-1$.
\end{proof}

We now investigate under which conditions $K$ can be $\{1\}$, $\{1,\omega\}$, or $\{1,-\omega\}$.  Now if $N$ is not a multiple of $4$, then by Proposition~\ref{prop:k0}, $K$ must be $\{1\}$, so assume $N$ is a multiple of $4$.

\begin{proposition}
Suppose $N$ is even.  Then $\abs(\omega)=\abs(-\omega)=h_1\in E$ (recall that $h_1=11\cdots 1$ is the word of all $1$s).
\end{proposition}
\begin{proof}
Recall that $\omega=\zeta_1\cdots\zeta_{N-1}$.  Then $\abs(\omega)=(1,\ldots,1)\in \F_2^{N-1}$.  Applying the isomorphism $(\pi|_E)^{-1}:\F_2^{N-1}\cong E$ described earlier, this corresponds to $h_1=(1,1,\ldots,1)\in E$.
\end{proof}

\begin{corollary}
If $h_1\not\in C$, then $K=\{1\}$.  If $h_1\in C$, then for every fermion $f$, either $\phi(\omega)(f)=f$ of $\phi(\omega)(f)=-f$.
\label{cor:homega}
\end{corollary}
\begin{proof}
If $\omega\in K$ or $-\omega\in K$, then $h_1=\abs(\omega)=\abs(-\omega)\in C$.

If $h_1\in C$, then we know that $\rho(\abs(\omega))(|f|)=|f|$, and so $
\phi(\omega(f))=f$ or $-f$.
\end{proof}

The question of whether or not $\omega\in K$ comes down to whether, for every fermion $f$, $\phi(\omega)(f)=f$.  Likewise $-\omega\in K$ if and only if, for every fermion $f$, $\phi(\omega)(f)=-f$.  In other words, when $h_1\in C$, the existence of a non-trivial kernel $K$ comes down to whether the signs obtained by applying $\phi(\omega)$ are consistent.  In principle this could be checked for every fermion in the Adinkra, but by the following proposition, we only need check each connected component of the Adinkra.

\begin{proposition}
Let $N$ be even.  If $f_1$ and $f_2$ are two fermions in the same connected component of an Adinkra, then $\phi(\omega)(f_1)=f_1$ if and only if $\phi(\omega)(f_2)=f_2$.  Likewise, $\phi(\omega)(f_1)=-f_1$ if and only if $\phi(\omega)(f_2)=-f_2$.
\label{prop:consistenth1}
\end{proposition}
\begin{proof}
Since $f_1$ and $f_2$ are in the same connected component, there is a path connecting $f_1$ to $f_2$, which corresponds to an element $g\in\G_{N-1}$ so that $\phi(g)(f_1)=(-1)^sf_2$.

Now $\omega$ is in the centre of $\G_{N-1}$.  So $\omega g = g\omega$.  Suppose $\phi(\omega)(f_1)=(-1)^tf_1$.  Then
\[\phi(\omega)(f_2)=\phi(\omega)((-1)^s \phi(g)(f_1))=(-1)^s\phi(g)(\phi(\omega)(f_1))=(-1)^{s+t}\phi(g)(f_1)=(-1)^tf_2.\]
\end{proof}

\begin{theorem}
Suppose $A$ is a connected Adinkra.  If $h_1\in C$, then either $K=\{1,\omega\}$ or $K=\{1,-\omega\}$.  If $h_1\not\in C$, then $K=\{1\}$.

For disconnected Adinkras, $K=\{1,\omega\}$ if and only if each connected component of the Adinkra has $K=\{1,\omega\}$.  Likewise, $K=\{1,-\omega\}$ if and only if each connected component of the Adinkra has $K=\{1,-\omega\}$.  We can only have $K=\{1\}$ if $h_1\not\in C$ (so that for some connected component, $K=\{1\}$) or if there are some connected components with $K=\{1,\omega\}$ and others with $K=\{1,-\omega\}$.
\end{theorem}
\begin{proof}
If $N$ is not a multiple of $4$, then $h_1\not\in C$ (since $C$ is doubly even) and $K=\{1\}$, and the theorem is proved.

Suppose $A$ is connected.  Let $f_0$ be a fermion of $A$.  If $h_1\in C$, then by Corollary~\ref{cor:homega}, $\phi(\omega)(f_0)=f_0$ or $\phi(\omega)(f_0)=-f_0$.

Suppose $\phi(\omega)(f_0)=f_0$.  By Proposition~\ref{prop:consistenth1}, since $A$ is a connected Adinkra, then for all fermions $f$, $\phi(\omega)(f)=f$ and $K=\{1,\omega\}$.  By a similar argument, if $\phi(\omega)(f_0)=-f_0$, $K=\{1,-\omega\}$.

Suppose $A$ is disconnected.  By restriction, $\omega\in K$ implies that for any connected component of $A$, $\omega\in K$.  Likewise, if $\omega\in K$ for each connected component of $A$, then for every fermion $f$, $\phi(\omega)(f)=f$.  Then $\omega\in K$ for $A$.  Likewise when $\omega$ is replaced by $-\omega$.

If $K=\{1\}$, then it is not the case that each connected component of the Adinkra has $K=\{1,\omega\}$, nor is it the case that each connected component has $K=\{1,-\omega\}$.  There conclusion follows.
\end{proof}

\begin{definition}
Let $A$ be an Adinkra.  Define $\chi_0$ for $A$ as follows:
\[
\chi_0=\begin{cases}
1,&\mbox{if $K=\{1,\omega\}$}\\
-1,&\mbox{if $K=\{1,-\omega\}$}\\
0,&\mbox{if $K=\{1\}$}
\end{cases}
\]
\end{definition}

\begin{remark}
This definition of $\chi_0$ generalizes that of Ref.~\cite{genomics1}, for connected Adinkras with $N=4$.  In that case, codes $C=\{0000\}$ and $C=\{0000,1111\}$ were considered.

When $\omega\in K$, then $h_1=1111\in C$, and for every fermion $f$, $\omega(f)=f$.  This means that a path beginning at $f$, following a colour sequence $(4,1,3,1,2,1)$ ends in at $f$ and has an even number of dashed edges.  Then by swapping the second and third colours, we see that the path from $f$ with colour sequence $(4,3,1,1,2,1)$ has an odd number of dashed edges.  The third and fourth colours cancel and gives us the colour sequence $(4,3,2,1)$.  This is the path that in Ref.~\cite{genomics1} was used to define $\chi_0=1$.  If $b$ is a boson, then a path starting at $b$ with the same colour sequence $(4,3,2,1)$ has an even number of dashed edges, according to the ideas in Ref.~\cite{cc}.

Likewise if $-\omega\in K$, the same argument shows that paths starting at a fermion with colour sequence $(4,3,2,1)$ have an even number of dashed edges, and paths starting at a boson with that colour sequence have an odd number of dashed edges, which is how Ref.~\cite{genomics1} defined $\chi_0=-1$.

If $K=\{1\}$, for connected Adinkras, $C=\{0000\}$, and Ref.~\cite{genomics1} defined $\chi_0$ in this case to be $0$.
\end{remark}

\subsection{Calculating $\mathcal{H}$}

\begin{theorem}
The signed monodromy group $\mathcal{H}$ is given by the following:
\begin{itemize}
\item If $K=1$, then $\mathcal{H}\cong \G_{N-1}$.

\item If $K\not=1$, then $\mathcal{H}\cong \G_{N-2}$.
\end{itemize}
\end{theorem}

\begin{proof}
We have generally that $\mathcal{H}\cong \G_{N-1}/K$.  If $K=1$ then $\mathcal{H}\cong \G_{N-1}$.  If $K\not=1$, then $K=\{1,\omega\}$ or $K=\{1,-\omega\}$, and by Proposition~\ref{prop:quotientomega}, $\mathcal{H}\cong \G_{N-2}$.
\end{proof}

\subsection{Calculating $\Sigma$}
We now turn our attention to $\Sigma$, which consists of those signed monodromies which give rise to trivial (unsigned) monodromies.  More generally, $\Sigma$ describes the extent to which an unsigned monodromy comes from many signed monodromies.

Note by the diagram (\ref{eqn:diagram}) that $\Sigma$ contains $\{\pm 1\}$ as a normal subgroup.

\begin{theorem}
The kernel $\Sigma$ of $\abs:\mathcal{H}\to\mathcal{M}$ is given by:
\begin{itemize}
\item If $K=1$, then $\Sigma\cong \F_2^{k+1}$.
\item If $K\not=1$, then $\Sigma\cong \F_2^k$.
\end{itemize}
\end{theorem}

\begin{proof}
Suppose $K=1$.  Then $\mathcal{H}\cong \G_{N-1}$ and $\phi$ is an isomorphism.  Under this isomorphism, $\Sigma$ is a normal subgroup of $\G_{N-1}$ that contains $-1$.  By the commutativity of the diagram, $\phi^{-1}(\Sigma)$ is the kernel of $\rho\circ\abs:\G_{N-1}\to\mathcal{M}$.  Since $C$ is the kernel of $\rho$, we have that the kernel of $\rho\circ\abs$ is $\abs^{-1}(C)$.

More explicitly, for each codeword $(x_1,\ldots,x_N)\in C$, there are two elements
\[\pm g_1^{x_2}\cdots g_{N-1}^{x_N}\]
in $\phi^{-1}(\Sigma)$, which becomes
\[\pm \zeta_1^{x_2}\cdots \zeta_{N-1}^{x_N}\]
in $\Sigma$.  All elements of $\Sigma$ are of this form, so that $|\Sigma|=2^{k+1}$.  By the fact that $C$ is doubly even, each such element squares to $1$.  Since doubly even codes are self-dual\cite{rCHVP}, it follows that any two such have an even number of factors in common, so that by Proposition~\ref{prop:commute}, $\Sigma$ is abelian.  These facts prove that $\Sigma\cong \F_2^{k+1}$.

Now suppose $K\not=1$.  Then $K=\{1,\omega\}$ or $K=\{1,-\omega\}$, and $\mathcal{H}\cong \G_{N-1}/K\cong \G_{N-2}$.

We begin as before by identifying the kernel of $\rho\circ\abs:\G_{N-1}\to \mathcal{M}$.  As before, this is $\abs^{-1}(C)$.  For every codeword $(x_1,\ldots,x_N)\in C$, we get two elements of this kernel of the form
\[\pm g_1^{x_2}\cdots g_{N-1}^{x_N}\]
Again, this set is an abelian group isomorphic to $\F_2^{k+1}$.

In this case, however, because of $K$, $\phi$ is not an isomorphism, and so to get $\Sigma$ we must quotient by $\omega$ or $-\omega$ (whichever is in $K$).  This shows that $\Sigma$ is an abelian group isomorphic to $\F_2^k$.

To consider more exactly how this fits in with $\mathcal{H}$, we trace this construction through the isomorphism in Proposition~\ref{prop:quotientomega}.  Use a generating set of $C$ so that at most one generator has $x_N=1$.  Let $C_0$ be the subcode that is generated by the other generators.  Then for every word $(x_1,\ldots,x_{N-1},0)\in C_0$, we have
\[\pm g_1^{x_2}\cdots g_{N-2}^{x_{N-1}}\]
in $\G_{N-2}$.  For $\mathcal{H}$, this is
\[\pm \zeta_1^{x_2}\cdots \zeta_{N-2}^{x_{N-1}}.\]

\end{proof}

To fit together our knowledge of $\Sigma$ and $K$ from a more abstract perspective, it will be helpful to apply the Snake Lemma\footnote{See Appendix~\ref{app:snake}} to the bottom two short exact sequences of our diagram (\ref{eqn:diagram}), which we show here:
\[
\xymatrix{
1\ar[r]&K\ar[r]\ar[d]^{\abs}&\G_{N-1}\ar[r]^{\phi}\ar[d]^{\abs}&\mathcal{H}\ar[r]\ar[d]^{\abs}&1\\
1\ar[r]&C \ar[r]&\F_2^{N-1} \ar[r]^{\rho}&\mathcal{M} \ar[r]&1\\
}
\]

The Snake Lemma then gives us the following exact sequence:
\[
\xymatrix{
1\ar[r]&\{\pm 1\} \ar[r]&
\Sigma \ar[r]&
C/\abs(K) \ar[r]&
1
}
\]
As we found earlier, $\Sigma$ is abelian and in fact either $\F_2^k$ or $\F_2^{k+1}$, so we can see that this sequence must split, and that
\[\Sigma\cong \{\pm 1\} \oplus C/\abs(K).\]

\begin{example}[$k=0$]  In this specific case, the code $C$ is trivial, i.e., $C=\{000\cdots 0\}$.  Then the commutative diagram becomes an isomorphism of short exact sequences:

\[
\xymatrix{
1\ar[r]\ar[d]^{\cong}&\{\pm 1\}\ar[r]\ar[d]^{\cong}&\G_{N-1}\ar[r]^{\abs}\ar[d]^{\cong}&\F_2^{N-1}\ar[r]\ar[d]^{\cong}&1\ar[d]^{\cong}\\
1\ar[r]&\Sigma \ar[r]&\mathcal{H} \ar[r]^{\abs}&\mathcal{M} \ar[r]&1\\
}
\]
(Note that we have rotated the diagram for typesetting reasons.)

Then we can view $\mathcal{H}$ as $\G_{N-1}$, and every signed monodromy is a monodromy with an extra $\pm 1$ sign.

The loops $w_1,\ldots,w_{N-1}$ give rise to signed monodromies $\zeta_1,\ldots,\zeta_{N-1}$, which generate $\mathcal{H}\cong \G_{N-1}$.  This is like the (unsigned) monodromies in $\mathcal{M}$, except that there are two signed monodromies for each unsigned monodromy, which differ due to an overall sign, which is influenced by the order in which the loops $w_i$ are traversed.  This overall sign is in $\Sigma\cong\{1,-1\}$.
\end{example}

\begin{example}[$N=4,\ C=\gen{(1111)}$ (connected Adinkra)]
In this case, $K=\{1,\omega\}$ or $K=\{1,-\omega\}$.  For this example, suppose we choose $K=\{1,\omega\}$.  The signed monodromy group $\mathcal{H}$ is generated by elements $\zeta_1,\zeta_2,\zeta_3$ and they have the same meaning as for the $4$-cube.  But since $\zeta_1\zeta_2\zeta_3=\omega$ is in $K$,  we also have $\zeta_3=\zeta_2\zeta_1$, with the result that the signed monodromy group is generated by $\zeta_1$ and $\zeta_2$, and $\mathcal{H}\cong \G_2\cong Q_8$.

\[
\xymatrix{
&&1\ar[d]\ar[r]&1\ar[d]\\
&1\ar[r]\ar[d]&\{\pm 1\}\ar[r]\ar[d]&\{\pm 1\}\ar[d]\\
1\ar[r]\ar[d]&\{1,\omega\}\ar[r]\ar[d]^{\abs}&\G_{3}\ar[r]^{\phi}\ar[d]^{\abs}&\mathcal{H}\cong \G_2 \ar[r]\ar[d]^{\abs}&1\ar[d]\\
1\ar[r]&\{0000,1111\} \ar[r]&\F_2^3 \ar[r]^{\rho}\ar[d]&\mathcal{M}\cong \F_2^2 \ar[r]\ar[d]&1\\
&&1\ar[r]&1&
}
\]
\end{example}

\begin{example}[$N=5,k=1$]
In this specific case, the kernel $K=1$. Let us consider the Adinkra obtained by quotienting the $5$-cube by the code $C=\gen{(11110)}$. 
In this case, the covering group is $\mathcal{M}=E/C$ is of order $2^{5-1-1}=8$. 
There are $8$ bosons and we have the following commutative diagram:
\[
\xymatrix{
&&1\ar[d]\ar[r]&1\ar[d]\\
&1\ar[r]\ar[d]&\{\pm 1\} \ar[r]\ar[d]&\{\pm 1,\pm\zeta_1\zeta_2\zeta_3\}\ar[d]\\
1\ar[r]\ar[d]&1\ar[r]\ar[d]^{\abs}&\G_{4}\ar[r]^{\phi}\ar[d]^{\abs}&\mathcal{H}\cong \G_4\ar[r]\ar[d]^{\abs}&1\ar[d]\\
1\ar[r]&\{00000, 11110\} \ar[r]&\F_2^4 \ar[r]^{\rho}\ar[d]&\mathcal{M}\cong \F_2^3 \ar[r]\ar[d]&1\\
&&1\ar[r]&1&
}
\]
The kernel $K$ is trivial, which means that the signed monodromy group is $\G_4$.  This group is generated by $\zeta_1, \ldots, \zeta_4$, which has $2^5=32$ elements (same as for the $5$-cube), but the monodromy group $\mathcal{M}$ is $\F_2^3$, which has $2^3=8$ elements.

The group $\Sigma$ has not only the usual $1$ and $-1$, but also $\zeta_1\zeta_2\zeta_3$ and $-\zeta_1\zeta_2\zeta_3$.  These correspond to traversing colours $1$, $2$, $3$, and $4$, and in terms of $\mathcal{M}$ sends every fermion to itself.  But in terms of the signed monodromy $\mathcal{H}$, 4 of the fermions are sent to their negatives\cite{cc}.
\end{example}

\section{Relations between $L_I$ matrices}
\label{sec:relations}
As one application of this, we consider relations between the various $L_I$ and $R_I$ matrices.

There are sometimes relations between $L_I$ and $R_I$ matrices.  There are consequences to the Garden algebra relations, for instance, $L_IR_I$ should be the identity, and $L_IR_J=-L_JR_I$, and so on.  But there are some relations that do not always happen in the Garden Algebra but nevertheless may happen in a specific representation.  For instance, when $N=4$, $k=1$, there is the example
\begin{align*}
L_1&=\left[\begin{array}{cccc}
1&0&0&0\\
0&1&0&0\\
0&0&1&0\\
0&0&0&1
\end{array}\right]\\
L_2&=\left[\begin{array}{cccc}
0&-1&0&0\\
1&0&0&0\\
0&0&0&1\\
0&0&-1&0
\end{array}\right]\\
L_3&=\left[\begin{array}{cccc}
0&0&0&-1\\
0&0&1&0\\
0&-1&0&0\\
1&0&0&0
\end{array}\right]\\
L_4&=\left[\begin{array}{cccc}
0&0&-1&0\\
0&0&0&-1\\
1&0&0&0\\
0&1&0&0
\end{array}\right]\\
\end{align*}
In addition, $R_1=L_1$, $R_2=-L_2$, $R_3=-L_3$, $R_4=-L_4$.

In this representation,
\[R_1L_2R_3=R_4\]
which is not generally true for arbitrary representations.

Note that as matrices, $L_1$ is the identity, so $L_1R_2=R_2$, but we do not consider such equations as relations because $L_1R_2$ goes from bosons to bosons, while $R_2$ goes from bosons to fermions.

\begin{theorem}
Non-trivial relations occur if and only if $h_1\in C$.
\end{theorem}

\begin{proof}

If $h_1\in C$, then either $\omega\in K$ or $-\omega\in K$.  Then in $\mathcal{H}$,
\[\pm 1 =\zeta_1\cdots \zeta_{N-1}\]
which can be written as
\begin{align*}
\pm 1
&=\rho_1\lambda_2\prod_{i=1}^{(N-2)/2}\rho_1\lambda_{2i+1}\rho_1\lambda_{2i+2}\\
&=\rho_1\lambda_2\prod_{i=1}^{(N-2)/2} -\rho_1\lambda_1\rho_{2i+1}\lambda_{2i+2}\\
&=(-1)^{(N-2)/2}\rho_1\lambda_2\prod_{i=1}^{(N-2)/2}\rho_{2i+1}\lambda_{2i+2},\\
\end{align*}
or as matrices,
\[R_1L_2R_3L_4\cdots R_{N-1}L_N= \pm 1\]
This can be written as
\[L_N=\pm L_{N-1}\cdots R_4L_3R_2L_1.\]
This is a nontrivial relation in the Garden algebra.

Conversely, if it were possible to write $L_N$ in terms of other matrices, this would result in a product of the form
\[R_{i_1}\cdots L_{i_k}=\pm 1\]
which in turn can be manipulated, via the ideas above, to
\[\pm\zeta_{i_1-1}\cdots \zeta_{i_k-1}= 1\]
and thus a non-trivial element of $K$.
\end{proof}


\section{Acknowledgments}
Acknowledgment is given by all of the authors for their participation in the second annual Brown University Adinkra Math/Phys Hangout during 2017 and supported by the endowment of the Ford Foundation Physics Professorship.

K. Iga and K. Stiffler were partially supported by the endowment of the Ford Foundation Professorship of Physics at Brown University.  K. Iga was partially supported by the U.S. National Science Foundation grant PHY-1315155.

J. Kostiuk would also like to acknowledge the NSERC Alexander Graham Bell Canada Graduate Scholarship Program for supporting him during that time period. 

\appendix
\section{Summary of Commutative Diagrams and Exact Sequences}\label{a:SES}
This appendix includes notions like commutative diagrams and exact sequences that are common in algebraic topology, algebraic geometry, homological algebra, and many other such subjects.  This is included to help readers who are less familiar with these subjects.  For more information, see Ref.~\cite{MunkresAT}.

If $A$ and $B$ are groups, then the diagram
\[
\xymatrix{
A\ar[r]^f& B
}
\]
denotes a group homomorphism $f$ with domain $A$ and codomain $B$.  In this paper, we sometimes see many of these put together, for instance, like this:
\[
\xymatrix{
A\ar[r]^f\ar[d]^g& B\ar[d]^i\\
C\ar[r]^h& D
}
\]
We say this diagram is {\em commutative} if $i\circ f = h\circ g$.

When two or more such homomorphisms are aligned collinearly,
\begin{equation}
\xymatrix{
\cdots\ar[r]^{f_0}&A_1\ar[r]^{f_1}& A_2\ar[r]^{f_2} &A_3\ar[r]^{f_3}&\cdots\\
}
\end{equation}
we say the sequence is {\em exact} if for every $i$, the image of $f_i$ is equal to the kernel of $f_{i+1}$.  Note that in that case, if for some $j$, $A_j$ is the trivial group $1$, then $f_{j}$ must be a monomorphism and $f_{j-1}$ is an epimorphism.

The term {\em short exact sequence} refers to an exact sequence of four maps where the first and last groups are trivial:
\begin{equation}\label{e:SESGen}
\xymatrix{
1\ar[r]& A \ar[r]^\alpha & B \ar[r]^{\beta}& C \ar[r]&1
}
\end{equation}
The important features of a short exact sequence are, in no particular order
\begin{enumerate}
	\item $\alpha$ is a monomorphism
	\item $\beta$ is an epimorphism
	\item $\ker(\beta) = \text{im}(\alpha)$
	\item $\ker(\beta)\cong A$
	\item $\cok(\alpha)\cong C$
\end{enumerate}

\subsection{Snake Lemma}
\label{app:snake}
In Section~\ref{sec:proof2}, we referred to the following standard lemma from the theory of commutative diagrams:
\begin{lemma}[Snake Lemma]
Given two short exact sequences in the following commutative diagram,
\[
\xymatrix{
&1\ar[r]&A_1 \ar[r]\ar[d]^{\phi_1}&A_2 \ar[r]\ar[d]^{\phi_2}&A_3 \ar[r]\ar[d]^{\phi_3}& 1\\
&1\ar[r]&B_1 \ar[r]&B_2 \ar[r]&B_3 \ar[r]& 1\\
}
\label{eqn:diagramsnake1}
\]
there is a long exact sequence
\[
\xymatrix{
1\ar[r]&\ker \phi_1 \ar[r]&
\ker \phi_2 \ar[r]&
\ker \phi_3 \ar[r]^\delta&
\cok \phi_1 \ar[r]&
\cok \phi_2 \ar[r]&
\cok \phi_3 \ar[r]&1
}\]
\end{lemma}

The proof is a matter of diagram chasing.  For instance, in Ref.~\cite{MunkresAT}, where it is called the serpent lemma, it is an exercise.

\bibliographystyle{plain}
\bibliography{susyrefs}

\end{document}